\documentclass[12pt,oneside,reqno]{amsart}
\usepackage[utf8]{inputenc}
\usepackage{amsaddr}
\usepackage{amsfonts}
\usepackage{amsmath}
\usepackage{amssymb}
\usepackage{amsthm}
\usepackage{graphicx}

\makeatletter
\@namedef{subjclassname@2020}{\textup{2020} Mathematics Subject Classification}
\makeatother

\theoremstyle{theorem}
\newtheorem{theorem}{Theorem}
\newtheorem{proposition}[theorem]{Proposition}
\newtheorem{lemma}[theorem]{Lemma}

\theoremstyle{definition}
\newtheorem{example}[theorem]{Example}

\newtheorem{remark}[theorem]{Remark}

\def\lock{\includegraphics[scale=1]{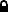}}
\def\unlock{\includegraphics[scale=1]{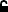}}
\newcommand{\IR}{\mathbb{R}}
\newcommand{\IN}{\mathbb{N}}
\newcommand{\bsone}{\mathbf{1}}
\newcommand{\bsonei}{\mathbf{1}_{I_j}}
\newcommand{\ZZ}{\mathbb{Z}}
\newcommand{\positive}{\Delta_+}
\newcommand{\negative}{\Delta_-}
\newcommand{\const}{\Delta_0}
\newcommand{\Deltafx}{\Delta f(i)}

\newcommand{\fun}{f_{\unlock}}
\newcommand{\flo}{f_{\lock}}
\newcommand{\fdiff}{f^{\rm diff}}
\newcommand{\Deltafdx}{\Delta f^{\rm diff}(i)}

\newcommand{\dd}{\mathrm{d}}

\DeclareMathOperator{\minimize}{\minimize}

\title{Unlocking Your Bike the Easy Way}
\author{Mathias Sonnleitner}
\address{Faculty of Computer Science and Mathematics, University of Passau, \\ 94032 Passau, Germany}
\email{mathias.sonnleitner@uni-passau.de}

\date{\today}

\begin{document}

\begin{abstract}
Combination locks are widely used to secure bicycles. We consider a combination lock consisting of adjacent rotating dials with the first nonnegative integers printed on each of them. Assuming that we know the correct combination and we start from an incorrect combination, what is the minimal number of steps to arrive at the correct combination if in each step we are allowed to turn an arbitrary number of adjacent dials once in a common direction? We answer this question using elementary methods and show how this is related to a variation of (multivariate) functions.
\end{abstract}

\maketitle

\section{Introduction.}

Let $n$ and $N$ be positive integers (i.e., in $\IN$) and consider a combination lock with $n$ dials, each showing the numbers $0,1,\dots,N-1$. By rotating the dials any combination 
\begin{equation} \label{eq:combination}
	f=\big(f(1),\dots,f(n)\big),\quad\text{where }f(i)\in [0,N-1]\text{ for every }i\in [n],
\end{equation}
can be transformed into any other. Here, we used $[m]=\{1,2,\dots,m\}$ for $m\in\IN$ and $[0,m]=\{0,1,\dots,m\}$ for $m\in\IN\cup \{0\}$. 

In order to reach the known correct combination $\fun$ starting from an incorrect combination $\flo$, we rotate dials repeatedly into one of two possible directions. Assuming that rotating any number of adjacent dials once into a common direction counts as one rotation, what is the minimal number of rotations needed to turn $\flo$ into $\fun$? 
\begin{figure}[h]
	\begin{center}
		\includegraphics[scale=1]{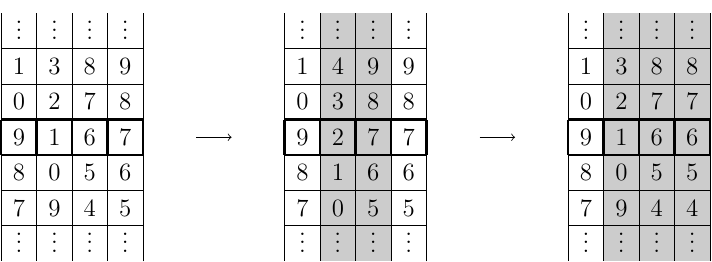}
	\end{center}
	\caption{\small{An inefficient way to turn combination $(9,1,6,7)$ into $(9,1,6,6)$ with rotated dials shaded in grey. The second combination is obtained from the first by rotating the second and the third dial downwards and the third combination is obtained from the second by rotating the second, third, and fourth dials upwards.}}
	\label{fig:example}
\end{figure}

Aside from the stackoverflow-post \cite{Leo13}, where an unsolved instance of this question with $N=26$ appears, we could not find any reference to this very intuitive problem. To state and solve it in a formal way requires some notation which we now introduce. 

We identify a rotation with an interval of dials $[a,b]=\{i\in [n]\colon a\le i \le b\}$, where $a,b\in [n]$, and a direction $\varepsilon\in \{1,-1\}$. Applying such a rotation to a combination $f$ as in display \eqref{eq:combination} yields the combination
\[
f'=\big(f'(1),\dots,f'(n)\big),\quad\text{where }
f'(i)=
f(i)+\varepsilon\bsone_{[a,b]}(i)\text{ mod } N, \quad i\in [n],
\]
where the indicator $\bsone_{[a,b]}(i)$ evaluates to one if $i\in [a,b]$ and to zero otherwise. Thus, depending on the direction $\varepsilon$, the digits of $f$ belonging to $[a,b]$ are either increased or decreased by one and then reduced modulo $N$. 

By using a suitable set of rotations, applied one after the other in any order, one can turn $\flo$ into any other combination of the form \eqref{eq:combination}. Such a set of rotations can be identified with a choice of $m\in\IN$ different intervals of dials $I_j=[a_j,b_j]$, $j\in [m]$, where $a_j,b_j\in [n]$, and numbers $c_j\in \ZZ$ which count how often exactly the dials in $I_j$ are rotated in total. Here, we subtract rotations in a negative direction from those in a positive direction and interpret a positive (negative) $c_j$ as $|c_j|$ rotations in positive (negative) direction. This ensures that rotations of the same dials in opposite directions cancel each other.

To answer the question stated above, we have to determine the minimal total number $\sum_{j=1}^{m}|c_j|$ of rotations needed to turn $\flo$ into $\fun$, that is, the solution to the following discrete optimization problem.

\begin{align} \label{eq:P1}
\text{Minimize}& \quad \sum_{j=1}^{m}|c_j|\nonumber\\
\text{subject to}& \quad \fun(i)\equiv \flo(i)+\sum_{j=1}^{m}c_j \bsone_{[a_j,b_j]}(i)\quad (\mathrm{mod}\, N), \quad i\in [n]\nonumber\\
&\quad m\in\IN, a_j,b_j\in [n] \text{ and }c_j\in\ZZ, \quad j\in [m].\tag{P1}
\end{align}

We can and do assume that $|c_j|\le N$ for all $j\in [m]$ since rotating the same dials $N$ times in the same direction has no effect. Thus, problem \eqref{eq:P1} admits a finite set of feasible solutions and may be solved through exhaustive search given $n,N,\fun$ and $\flo$. Instead of using more advanced techniques from discrete optimization, see, e.g., \cite{PR88} for an overview, we present an illustrative elementary approach to solve \eqref{eq:P1}.

The solution of \eqref{eq:P1} and thus the answer to our question will depend only on the difference between $\fun$ and $\flo$ which is the combination
\[
\fdiff(i)=\fun(i)-\flo(i) \mod N,\quad \text{}i\in [n],
\]
and in fact only on its jumps which are measured by its (forward) difference
\[
\Delta \fdiff\colon [0,n]\to \ZZ,\quad 
i\mapsto
\fdiff(i+1)-\fdiff(i),
\]
where we extended $\fdiff$ by setting $\fdiff(0)=\fdiff(n+1)=0$. Define the sets 
\[
\positive=\{i\in [0,n]\colon \Delta \fdiff(i)>0\}\quad \text{and}\quad
\negative=\{i\in [0,n]\colon \Delta \fdiff(i)<0\},
\]
of upward jumps and downward jumps which are both nonempty. We enumerate these nonincreasingly according to the size of the jump, i.e.,
\[
\positive=\{i_1^{+},\dots,i_{\#\positive}^+\}\quad\text{such that}\quad|\Delta \fdiff(i_1^+)|\ge \cdots \ge |\Delta \fdiff(i_{\#\positive}^+)|,
\]
where $\#\positive$ denotes the cardinality of $\positive$ and for $\negative$ we replace $+$'es by $-$'es. By summing over all jumps and grouping large upward with large downward jumps we obtain the following statement on the minimal number of rotations.

\begin{theorem} \label{thm:main}
The optimal value in \eqref{eq:P1} is given by
\begin{equation} \label{eq:formula}
	\frac{1}{2}\sum_{i=0}^{n}|\Delta \fdiff(i)|-\sum_{k=1}^{K}\big( |\Delta \fdiff(i_{k}^+)|+|\Delta \fdiff(i_{k}^-)|-N\big)_+,
\end{equation}
where $K= \min\{\#\positive,\#\negative\}$ and $x_+=\max\{x,0\}$.
\end{theorem}

Formula~\eqref{eq:formula} in Theorem~\ref{thm:main} gives the optimal value of Problem~\eqref{eq:P1} and thus the minimal number of rotations needed to turn $\flo$ into $\fun$, or equivalently, to turn $\fdiff$ into $(0,\dots,0)$. This answers the question raised in the beginning.

Before we prove Theorem~\ref{thm:main}, let us present two examples with $N=10$. 

\begin{example}\label{ex:1}
Let 
\[
\fdiff=(9,0,9,0,9)\quad\text{with}\quad\Delta \fdiff=(9,-9,9,-9,9,-9).
\]
It is not too hard to see that three rotations turning each nine to a zero suffice to turn $\fdiff$ to zero and that Theorem~\ref{thm:main} gives the same answer. 
\end{example}

\begin{example}\label{ex:2}
Let 
\[
\fdiff=(9,8,\ldots,1)\quad\text{with}\quad\Delta \fdiff=(9,-1,-1,\dots,-1).
\]
Rotate such that all non-zero entries of $\fdiff$ are increased by one ($9\rightarrow 0$). Doing this nine times we arrive at the zero combination and by Theorem~\ref{thm:main} this is the best possible. 
\end{example}

Comparing these examples, each rather extreme in its own way, we see that for Example~\ref{ex:1} both sides of the minus sign in \eqref{eq:formula} are rather large but roughly of the same size and compensate for each other. For Example~\ref{ex:2}, the sum on the left is not too large and the sum on the right vanishes. 

As a first step toward the proof of Theorem~\ref{thm:main} we give an interpretation of Problem~\eqref{eq:P1} in terms of efficiently building functions from indicators. For this purpose, consider combinations as functions from $[n]$ to $[0,N-1]$ which can be added and subtracted pointwise. If we subtract $\flo$ from both sides of the equality in \eqref{eq:P1}, then Problem \eqref{eq:P1} is equivalent to finding a representation of the form $\fdiff\equiv\sum_{j=1}^{m}c_j \bsone_{[a_j,b_j]}\ (\text{mod } N)$ with minimal cost $\sum_{j=1}^{m}|c_j|$ when each function $c_j\bsone_{[a_j,b_j]}$ costs $|c_j|$. Then this minimal cost is given by formula \eqref{eq:formula} in Theorem~\ref{thm:main}.

It will be helpful to extend this interpretation to more general functions as follows. For any $f\colon [n]\to \ZZ$ define its cost by
\begin{equation} \label{eq:budget}
C(f)
=\min\Big\{\sum_{j=1}^{m}|c_j|\colon f=\sum_{j=1}^{m}c_j\bsone_{[a_j,b_j]} \text{ where }m,a_j,b_j\text{ and }c_j \text{ are as in }\eqref{eq:P1}\Big\}. 
\end{equation}
Since for any $a,b\in\ZZ$ we have $a\equiv b\ (\mathrm{mod}\, N)$ if and only if $a+cN=b$ for some $c\in\ZZ$, Problem~\eqref{eq:P1} is equivalent to
\begin{align} \label{eq:P2}
\text{minimize}& \quad C(\fdiff+Nk)\nonumber\\
\text{subject to}& \quad k\colon [n]\to\ZZ \tag{P2},
\end{align}
where ``equivalent'' means that a solution to \eqref{eq:P2} gives rise to a solution of \eqref{eq:P1}, and vice versa, and that the optimal values are the same. 

We can split \eqref{eq:P2} into two subproblems, where the first consists of computing $C(f)$ for a general $f\colon [n]\to\ZZ$ and the second of finding the optimal $k$ in \eqref{eq:P2}. The first subproblem is resolved by the following proposition which is interesting on its own. 

\begin{proposition}\label{pro:budget}
For $f\colon [n]\to\ZZ$ we have
\begin{equation} \label{eq:pro-budget}
C(f)
= \frac{1}{2}\sum_{i=0}^{n}|\Deltafx|,
\end{equation}
where we set $f(0)=f(n+1)=0$.
\end{proposition}

The proof of Proposition~\ref{pro:budget} is given in the upcoming section as part of the proof of Theorem~\ref{thm:main}. From the proofs we deduce an optimal set of rotations solving problem \eqref{eq:P1}, thus providing its full solution. In Section~\ref{sec:variation} we connect Proposition~\ref{pro:budget} to existing results on different concepts of variation.

\section{Proofs and optimal rotations.}

We begin by proving an extended version of Proposition~\ref{pro:budget}. To state it, define for any $x\in\ZZ$ its positive part $x_+=\max\{x,0\}$ and its negative part $x_-=(-x)_+$ such that $x=x_+-x_-$ and $|x|=x_+ + x_-$. Intuitively, it expresses the cost as half the total sum of upward and downward jumps which is equal to the sum over all upward or downward jumps.

\begin{proposition}\label{pro:budget-general}
With $C(f)$ as in \eqref{eq:budget}, for any $f\colon [n]\to\ZZ$ we have 
\[
C(f)
= \frac{1}{2}\sum_{i=0}^{n}|\Deltafx|
=\sum_{i=0}^{n}\Deltafx_+
=\sum_{i=0}^{n}\Deltafx_-,
\]
where we set $f(0)=f(n+1)=0$.
\end{proposition}
\begin{proof}
The second and the third equality follow from
\[
0
=\sum_{i=0}^{n}\Deltafx
=\sum_{i=0}^{n}\Deltafx_+-\sum_{i=0}^{n}\Deltafx_-.
\]

Next, we prove the upper bound of $C(f)$. The idea behind the proof is to write $f=f_+-f_-$ with $f_+(i)=f(i)_+$ and $f_-(i)=f(i)_-$ for $i\in [n]$ and to represent the nonnegative functions $f_+$ and $f_-$ as in \eqref{eq:budget}. This is done by partitioning superlevel sets into connected components and summing the corresponding indicator functions as indicated in Figure~\ref{fig:pro-example}. 
\begin{figure}[h]
	\begin{center}
		\includegraphics[scale=1]{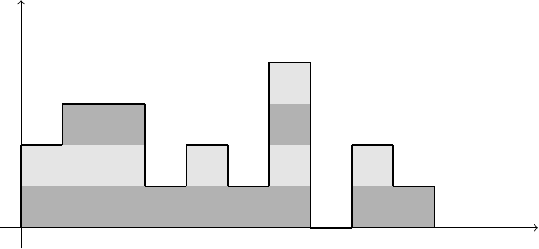}
	\end{center}
	\caption{\small{A nonnegative function $f\colon [n]\to \ZZ$ is graphically represented as a linear combination of indicator functions. Each shaded block corresponds to the interval it spans horizontally and is shaded according to whether this interval belongs to a superlevel set $\{f\ge \ell\}$ of odd or even level $\ell\in\IN$. Alternatively, one can view $C(f)$ as the minimal number of blocks needed to build $f$ and decomposing the superlevel sets provides a formal way of doing this in an optimal way. Intuitively, going from left to right, one can create new blocks at upward jumps and close them at downward jumps.}}
	\label{fig:pro-example}
\end{figure}

Assume first that $f\ge 0$, which applies to $f_+$ and $f_-$. Then for each level $\ell\in\IN$ with $\ell\le \sup f=\sup_{i\in [n]} f(i)$ we find $M_{\ell}\in\IN_0$ and $a_{\ell,k}, b_{\ell,k}\in [n]$ for $k\in [M_{\ell}]$ such that
\[
\{i\in [n]\colon f(i)\ge \ell\}
=\bigcup_{k=1}^{M_{\ell}}[a_{\ell,k},b_{\ell,k}]
\]
is a disjoint union of intervals of maximal length. The endpoints of the intervals are given by the location of upward and downward jumps, i.e., 
\begin{align*}
\big\{a_{\ell,k}\colon k\in [M_{\ell}]\big\} 
&=\big\{i\in[n]\colon f(i-1)<\ell\le f(i)\big\},\\
\big\{b_{\ell,k}\colon k\in [M_{\ell}]\big\} 
&=\big\{i\in[n]\colon f(i+1)<\ell\le f(i)\big\}
\end{align*}
and if $M_{\ell}\ge 2$ we assume without loss of generality that they are ordered such that 
\begin{equation} \label{eq:interlacing-ell}
	a_{\ell,k}< b_{\ell,k}+1 <a_{\ell,k+1}<b_{\ell,k+1}+1\quad \text{for every}\quad k\in [M_{\ell}-1].
\end{equation}
From the right-hand side of the pointwise equalities
\[
f
=\sum_{\ell=1}^{\sup f}\bsone_{\{i\in [n]\colon f(i)\ge \ell\}}
=\sum_{\ell=1}^{\sup f}\sum_{k=1}^{M_{\ell}}\bsone_{[a_{\ell,k},b_{\ell,k}]},
\]
we can collect equal summands, if there are any, to obtain a representation as in \eqref{eq:budget}. In any case, the cost is $\sum_{\ell=1}^{\sup f}M_{\ell}$ and thus
\[
M_{\ell}
=\#\{i\in[n]\colon f(i-1)<\ell\le f(i)\}
=\#\{i\in[n]\colon \ell\in (f(i-1),f(i)]\}
\]
yields the bound
\begin{equation} \label{eq:discrete-harman}
C(f)
\le\sum_{\ell=1}^{\sup f}M_{\ell}
=\sum_{\ell=1}^{\sup f}\sum_{i\in[n]}\bsone_{(f(i-1),f(i)]}(\ell)
=\sum_{i=1}^n\Delta f(i-1)_+.
\end{equation}
This establishes the upper bound of $C(f)$ if $f\ge 0$. 

In general, write $f=f_+-f_-$ with $f_+,f_-\ge 0$. By the first part of the proof we can write $f_+=\sum_{j=1}^{m^+}c_j^+ \bsone_{I_j^+}$ as in \eqref{eq:budget} with $m^+\in\IN, c_j^+\in\ZZ$ and intervals $I_j^+\subseteq [n]$ satisfying
\[
\sum_{j=1}^{m^+}|c_j^+|
\le \frac{1}{2}\sum_{i=0}^{n}|f_+(i+1)-f_+(i)|,
\]
and analogously for $f_-$. Let
\[
I_j=I_j^+\text{ and } c_j=c_j^+\text{ for }j\in [m^+]\quad\text{and}\quad I_{j+m^+}=I_j^-\text{ and } c_{j+m^+}=-c_j^-\text{ for }j\in [m^-].
\]
With $m=m^+ +m^-$ we obtain the pointwise equalities
\[
f
=f_+ - f_-
=\sum_{j=1}^{m^+}c_j^+ \bsone_{I_j^+} -\sum_{j=1}^{m^-}c_j^-\bsone_{I_j^-}
=\sum_{j=1}^{m}c_j \bsone_{I_j}.
\]
Further,
\begin{align*}
C(f)
\le \sum_{j=1}^{m}|c_j|
&\le\frac{1}{2}\sum_{i=0}^{n}| f_+(i+1)-f_+(i)| + \frac{1}{2}\sum_{i=0}^{n}|f_-(i+1)-f_-(i)|\\
&=\frac{1}{2}\sum_{i=0}^{n}|f(i+1)-f(i)|
\end{align*}
since $|x_+-y_+|+|x_- -y_-|=|x-y|$ for every $x,y\in\IR$. This completes the proof of the upper bound of $C(f)$ in the general case.

For the lower bound let $m\in\IN$, $c_j\in\ZZ$, and $I_j=[a_j,b_j]\subseteq [n]$ for $j\in [m]$ be such that
\begin{equation} \label{eq:lower-representation}
	f=\sum_{j=1}^{m}c_j\bsonei.
\end{equation}
It is sufficient to show
\begin{equation} \label{eq:lower-goal}
	\sum_{j=1}^{m}|c_j|
	\ge \sum_{i=0}^{n}\Deltafx_+.
\end{equation}

Trivially, this holds if right-hand side of \eqref{eq:lower-goal} equals zero and in particular if the set
\[
\Delta^+_0
=\{i\in [0,n]\colon \Deltafx\ge 0\}
\]
is empty. Assume that $\Delta^+_0$ is nonempty and, as for the superlevel sets above, find $M$ and $a_k',b_k'\in [0,n]$ for $ k\in [M]$ such that
\begin{equation} \label{eq:decomposition}
\Delta^+_0
=\bigcup_{k=1}^M [a_k',b_k']
\end{equation}
is a disjoint union of intervals of maximal length. Again, if $M\ge 2$, we assume without loss of generality that the endpoints of the intervals are ordered such that
\begin{equation} \label{eq:interlacing}
	a_k'< b_k'+1 <a_{k+1}'<b_{k+1}'+1\quad \text{for every}\quad k\in [M-1].
\end{equation}
Then the right-hand side of \eqref{eq:lower-goal} equals
\[
 \sum_{i\in \Delta^+_0}\Deltafx_+
= \sum_{k=1}^M\sum_{x=a_k'}^{b_k'} \Deltafx
= \sum_{k=1}^M f(b_k'+1)-f(a_k').
\]
With an interchange of sums, the representation \eqref{eq:lower-representation} yields 
\[
\sum_{k=1}^M f(b_k'+1)-f(a_k')
=\sum_{j=1}^{m} c_j\Big(\sum_{k=1}^M\bsonei(b_k'+1)-\sum_{k=1}^M\bsonei(a_k')\Big).
\]
After applying absolute values and the triangle inequality we arrive at
\[
\Big|\sum_{k=1}^M f(b_k'+1)-f(a_k')\Big|
\le \sum_{j=1}^{m} |c_j|\Big|\sum_{k=1}^M\bsonei(b_k'+1)-\sum_{k=1}^M\bsonei(a_k')\Big|.
\]
In the case of $M=1$ the proof of \eqref{eq:lower-goal} is already complete and if $M\ge 2$ the interlacing of the $a_k'$ and $(b_k'+1)$ as in \eqref{eq:interlacing} gives 
\[
\Big|\sum_{k=1}^M\bsonei(b_k'+1)-\sum_{k=1}^M\bsonei(a_k')\Big|
\le 1,\quad j\in [m].
\]
This proves the lower bound \eqref{eq:lower-goal} and completes the proof.
\end{proof}
 
Thus, also Proposition~\ref{pro:budget} is now proven. Together with the linearity of the difference $\Delta$ it implies that problem \eqref{eq:P2} can be restated as 
\begin{align} \label{eq:P2'}
\text{minimize}& \quad \frac{1}{2}\sum_{i=0}^{n}|\Deltafdx + N\Delta k(i)|\nonumber\\
\text{subject to}& \quad k\colon [0,n+1]\to\ZZ \text{ with } k(0)=k(n+1)=0. \tag{P2}
\end{align}

As $\Delta k$ appears in the objective function of \eqref{eq:P2'}, it will be helpful to optimize over $\Delta k$ instead of $k$. We observe that any $k\colon [n]\to\ZZ$ as in \eqref{eq:P2'} is determined by its difference $\Delta k\colon [0,n]\to \ZZ$ and that any function $k'\colon [0,n]\to \ZZ$ is equal to $\Delta k\colon [0,n]\to\ZZ$ for some $k\colon [n]\to\ZZ$ as in \eqref{eq:P2'} if and only if $\sum_{i=0}^{n}k'(i)=0$. Therefore, \eqref{eq:P2'}, and thus \eqref{eq:P1}, are equivalent to
\begin{align} 
\text{minimize}& \quad \frac{1}{2}\sum_{i=0}^{n}|\Deltafdx + Nk'(i)|\nonumber\\
\text{subject to}& \quad k'\colon [0,n]\to \ZZ,  \quad \sum_{i=0}^{n}k'(i)=0. \tag{P3}\label{eq:P3}
\end{align}

The following lemma on solutions to \eqref{eq:P3} is a consequence of
\[
|\Deltafdx|\le N-1,\quad i\in [0,n].
\]
\begin{lemma} \label{lem:k-abs-one}
	If $k'$ solves \eqref{eq:P3}, then $|k'(i)|\le 1$ for all $i\in [0,n]$.
\end{lemma}
\begin{proof}
	Suppose that $k'$ solves \eqref{eq:P3}. If $k'(i_1)> 1$ for some $i_1\in [0,n]$, then also $k'(i_2)\le -1$ for some $i_2\in [0,n]$. A case distinction shows that 
\[
|\Delta \fdiff(i_1) + N(k'(i_1)-1)|+|\Delta \fdiff(i_2) + N(k'(i_2)+1)|
\]
is strictly smaller than
\[
|\Delta \fdiff(i_1) + Nk'(i_1)|+|\Delta \fdiff(i_2) + Nk'(i_2)|.
\]
This contradicts the minimality of $k'$. The case $k'(i_1)<-1$ is analogous.
\end{proof}

Lemma~\ref{lem:k-abs-one} shows that for any function $k'$ with $\sum_{i=0}^{n}k'(i)=0$ appearing in \eqref{eq:P3} we can (and do so from now on) replace the codomain $\ZZ$ by $\{-1,0,1\}$ without changing the solution of \eqref{eq:P3}. For such a $k'$ there are as many $i^+\in [0,n]$ with $k'(i^+)=-1$ as $i^-\in [0,n]$ with $k'(i^-)=1$ (the choice of signs will become clear in a moment). By pairing these points, we can associate to $k'$ a set of pairs 
\begin{equation} \label{eq:pairs}
P\subseteq [0,n]^2 \text{ with } (i_1^+,i_1^-)\neq (i_2^+,i_2^-)\in P \Rightarrow \{i_1^+,i_1^-\}\cap \{i_2^+,i_2^-\}=\emptyset,
\end{equation}
i.e., any two different pairs from $P$ do not have a common element. Vice versa, to any $P$ we can associate a function $k'$ as in \eqref{eq:P3} by setting $k'(i^+)=-1$ and $k'(i^-)=1$ for every pair $(i^+,i^-)\in P$ and $k'(i_0)=0$ if $i_0\in P_0$, where 
\[
P_0
=\big\{i_0\in [0,n]\colon (i_0,i)\not\in P\wedge (i,i_0)\not\in P\text{ for all }i\in [0,n]\big\}
\]
is the set of $i_0\in [0,n]$ not belonging to any pair of $P$. This establishes (a non-bijective) correspondence between sets of pairs $P$ as in \eqref{eq:pairs} and functions $k'$ as in \eqref{eq:P3} with $\ZZ$ replaced by $\{-1,0,1\}$. Consequently, a change of variables allows us to restate \eqref{eq:P3} in terms of $P$ instead of $k'$. As a preparation, define,  for any $i^+,i^-\in [0,n]$,
\begin{align} \label{eq:G-def}
G(i^+,i^-)
=&|\Delta \fdiff(i^+)|+|\Delta \fdiff(i^{-})|\nonumber\\ 
&-\big(|\Delta \fdiff(i^+)-N|+|\Delta \fdiff(i^-)+N|\big),
\end{align}
which is the gain in the objective function if $(i^+,i^-)\in P$ instead of $k'(i^+)=k'(i^-)=0$. Then \eqref{eq:P3} and thus \eqref{eq:P1} are equivalent to 

\begin{align} 
\text{minimize}& \quad  \frac{1}{2}\Big(\sum_{i=0}^{n}|\Delta \fdiff(i)|-\sum_{(i^+,i^-)\in P}G(i^+,i^-)\Big)\nonumber\\
\text{subject to}& \quad P \text{ is as in \eqref{eq:pairs}}. \tag{P4}\label{eq:P4}
\end{align}

Solving \eqref{eq:P4} is equivalent to finding a set of pairs $P$ such that the total gain $\sum_{(i^+,i^-)\in P}G(i^+,i^-)$ is maximal. We first analyze the function $G$ and partition $[0,n]$ into the sets 
\begin{align*}
\positive&=\{i\colon \Deltafdx>0\},\quad\\
\negative&=\{i\colon \Deltafdx<0\}\quad \text{and}\quad\\
\const&=\{i\colon \Deltafdx=0\}.
\end{align*}
\begin{lemma} \label{lem:G-positive}
Let $i^+,i^-\in [0,n]$. If $i^+\in \negative\cup \const$ or $i^-\in\positive \cup \const$, then we have $G(i^+,i^-)<0$. Else, i.e., if $i^+\in\positive$ and $i^-\in\negative$, then 
\begin{equation} \label{eq:G-pos-neg}
G(i^+,i^-)=2\big(|\Delta \fdiff(i^+)|+|\Delta \fdiff(i^-)|-N\big).
\end{equation}
\end{lemma}
\begin{proof}
Suppose that $i^+\in \negative\cup \const$. Then $|\Delta \fdiff(i^-)|\le N-1$ implies
\[
G(i^+,i^-)
=|\Delta \fdiff(i^-)|-N-|\Delta \fdiff(i^-)+N|
\le -2
<0.
\]
The case $i^-\in \positive\cup \const$ is analogous and equality \eqref{eq:G-pos-neg} follows from a brief calculation.
\end{proof}

By Lemma~\ref{lem:G-positive} the gain $G(i^+,i^-)$ will be negative, except if $(i^+,i^-)\in \Delta_+\times \Delta_-$, which is nonempty if and only if $\fdiff\neq 0$. Consequently, any solution $P$ to problem \eqref{eq:P4} only includes pairs $(i^+,i^-)\in \Delta_+\times \Delta_-$, if any at all, and thus we can assume without loss of generality that $P\subseteq \positive\times \negative$. Then 
\begin{equation} \label{eq:P-opt}
\sum_{(i^+,i^-)\in P}G(i^+,i^-)
= 2\sum_{(i^+,i^-)\in P}\big(|\Delta\fdiff(i^+)|+|\Delta\fdiff(i^-)|-N\big).
\end{equation}

In order to maximize \eqref{eq:P-opt} (and minimize the objective function in \eqref{eq:P4}), we choose pairs $(i^+,i^-)\in \Delta_+\times \Delta_-$ maximizing $|\Delta\fdiff(i^+)|+|\Delta\fdiff(i^-)|$ as long as this sum is at least $N$. For this, enumerate $\Delta_+$ nonincreasingly according to the size of $|\Delta \fdiff(i^+)|$, i.e.,
\[
\positive=\{i_1^+,\dots,i_{\#\positive}^+\}\quad\text{such that}\quad|\Delta \fdiff(i_1^+)|\ge \cdots\ge|\Delta \fdiff(i_{\#\positive}^+)|
\]
and analogously for $\negative$ instead of $\positive$. Then the optimal $P$ will consist of the pairs $(i_1^+,i_1^-),\dots,(i_{K}^+,i_{K}^-)$, where $K\in\IN_0$ is maximal with $K \le \min\{\#\positive,\#\negative\}$ and
\begin{equation} \label{eq:gains}
|\Delta \fdiff(i_k^+)|+|\Delta \fdiff(i_k^-)|-N\ge 0,\quad k\in [K].
\end{equation}
If $K=0$ we set $P=\emptyset$. Finally, if $k\in [K+1,\min\{\#\positive,\#\negative\}]$, then the positive part of expression in \eqref{eq:gains} is zero.

This proves Theorem~\ref{thm:main}. \hfill $\square$

\begin{remark}
	As a consequence of the proof of Theorem~\ref{thm:main}, it does not matter whether we allow rotations with $c_j\in\IR$ instead of $c_j\in \ZZ$. This means the additional freedom of being able to rotate by an arbitrary fraction does not yield an improvement.	
\end{remark}

	So far we have only given the optimal value of \eqref{eq:P1} and in the following we show how a full solution can be found using the proof of Theorem~\ref{thm:main}. In general, it is not unique.
	
	In order to obtain optimal parameters $m,a_j,b_j$ and $c_j$ for Problem~\eqref{eq:P1}, we backtrack our steps through Problems \eqref{eq:P4},\eqref{eq:P3} and \eqref{eq:P2}. Looking at Problem~\eqref{eq:P4} we find an optimal set of pairs $P$ and consequently an optimal $k'$ in Problem~\eqref{eq:P3}. With the notation as in the proof of Theorem~\ref{thm:main}, for such a $k'$ its value $k'(i)$ is equal to $\pm 1$ if and only if $i\in \{i_1^{\pm},\dots,i_K^{\pm}\}$. From such an optimal $k'$ we can construct an optimal solution $k$ to \eqref{eq:P2} by setting 
\[
k(1)=k'(0)\quad \text{and}\quad k(i)=k(i-1)+k'(i-1), \quad i\in [n],
\]
or equivalently, 
\[
k(i)
=\sum_{j=0}^{i-1}k'(j)
=\#\{k\in [K]\colon i_k^-\le i-1\}-\#\{k\in [K]\colon i_k^+\le i-1\},\quad i\in [n].
\]
After defining $g(i)=\fdiff(i)+k(i)N, i\in [n]$, we can enter the proof of Proposition~\ref{pro:budget} to find the required optimal representation. 

More precisely, set $g(0)=g(n+1)=0$ and write $g=g_+-g_-$. For each $\ell\ge 1$ decompose the level sets of $g_+$ and $g_-$ into a disjoint union of maximal intervals as in
\[
\{g_+\ge \ell\}
=\bigcup_{j=1}^{M_{\ell}^+}[a_{\ell,j}^+,b_{\ell,j}^+]\quad\text{and}\quad
\{g_-\ge \ell\}
=\bigcup_{j=1}^{M_{\ell}^-}[a_{\ell,j}^-,b_{\ell,j}^-].
\]
Ignoring empty intervals, $m$ will be the number of distinct intervals $I_j$ of the form $[a_{\ell,j}^{\pm},b_{\ell,j}^{\pm}]$ appearing in  
\begin{equation} \label{eq:optimal-rep}
\sum_{\ell=1}^{\sup g_+}\sum_{j=1}^{M_{\ell}^+}\bsone_{[a_{\ell,j}^+,b_{\ell,j}^+]}
-\sum_{\ell=1}^{\sup g_-}\sum_{j=1}^{M_{\ell}^-}\bsone_{[a_{\ell,j}^-,b_{\ell,j}^-]}
\end{equation}
and for each interval $I_j$ we set $c_j$ to be the number of times it appears on the left-hand side minus the number of times it appears on the right-hand side in \eqref{eq:optimal-rep}. In this way, we obtain a solution to Problem~\eqref{eq:P1} with optimal value given by Theorem~\ref{thm:main}. 

An example is illustrated in Figure~\ref{fig:solution-example}. As can be seen from the previous Examples~\ref{ex:1} and \ref{ex:2}, the right-hand sum in formula~\ref{eq:formula} in Theorem~\ref{thm:main} compensates for big jumps of opposite sign, a fact which is represented in the choice of $k$ in Figure~\ref{fig:solution-example}.

\begin{figure}[h]
	\begin{center}
		\includegraphics[scale=1]{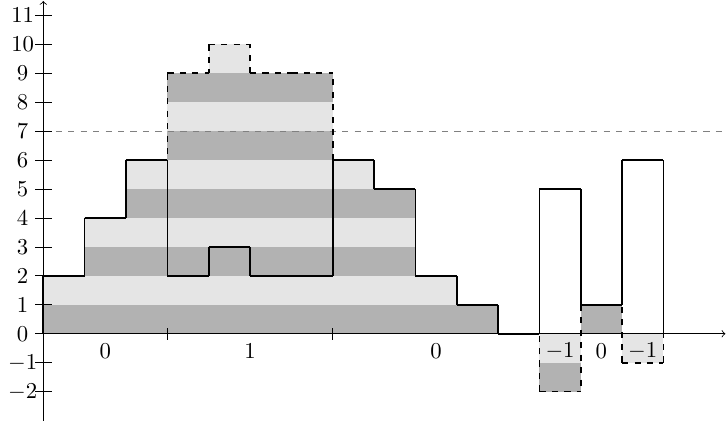}
	\end{center}
	\caption{\small{For $n=15$ and $N=7$ the combination $\fdiff=(2,4,6,2,3,2,2,6,5,2,1,0,5,1,6)$ is depicted using bold continuous lines. Below the horizontal axis the values of an optimal $k$ are indicated and in dashed lines the function $\fdiff+Nk$ is drawn. In a similar fashion as in Figure~\ref{fig:pro-example} the shaded blocks are used to indicate the optimal representation for the function $\fdiff+kN$ which attains the optimal value in \eqref{eq:P1}.  In this case we have }
	$
	\positive=\{5,5,4,2,2,2,1\}\text{, }\negative=\{-6,-4,-4,-3,-1,-1,-1,-1\}
	$
	\small{and $P$ consists of three pairs with total gain $7$. Thus, formula~\eqref{eq:formula} in Theorem~\ref{thm:main} gives $21-7=14$, which is exactly the number of shaded blocks.} 
	}
	\label{fig:solution-example}
\end{figure}

\begin{remark}
	The example in Figure~\ref{fig:solution-example} was generated using (pseudo-)random numbers. It would be interesting to study the minimal number of rotations, that is the quantity in formula~\eqref{eq:formula} in Theorem~\ref{thm:main}, for uniformly distributed random combinations. In the case $N=2$ this random quantity $R_{n,N}$ is related to the number of sign changes in a random bit string. Numerical experiments show that the expected value behaves as $\mathbb{E}R_{n,N}\approx \frac{N(n+1)}{8}$.
\end{remark}

\section{Variation of functions and discrepancy of points.}
\label{sec:variation}

Proposition~\ref{pro:budget} and its generalized Proposition~\ref{pro:budget-general} played essential roles in the proof of Theorem~\ref{thm:main}. In order to place the solved problem in a broader context, we shall present different concepts of variation of functions and statements which are related to Proposition~\ref{pro:budget}. 

The total variation of a real-valued function $f$ defined on some interval $[a,b]\subseteq\IR$ is
\[
V(f)
=\sup_{m\in\IN}\sup_{a\le x_1<\cdots<x_m\le b}\sum_{j=1}^{m-1}|f(x_{i+1})-f(x_{i})|,
\]
that is the supremum of the sum of absolute differences over any increasing sequence in $[a,b]$. Similar to the role of $C(f)$ in quantifying the necessary cost to represent a discrete function as a sum of indicators, the total variation $V(f)$ is a measure of the inherent complexity of $f$, for example with respect to noise reduction in images \cite{ROF92} or nonlinear approximation \cite[Sec. 3.2]{Dev98}.

To explain the relation between $C(f)$ and $V(f)$, let $g\colon [n]\to \ZZ$ and extend it to a piecewise constant function $g^{\rm ext}$ by setting it to zero on $\ZZ\setminus [n]$ and rounding to the next-largest integer. If we rescale by replacing $g^{\rm ext}$ with $g^{\rm ext}(n\cdot)$, then $g^{\rm ext}$ is supported in $[0,1]$, has a finite number of discontinuities and total variation with respect to any larger interval equal to
\begin{equation} \label{eq:variation}
V(g^{\rm ext})
=\sum_{i=0}^{n}|g^{\rm ext}\Big(\frac{i+1}{n}\Big)-g^{\rm ext}\Big(\frac{i}{n}\Big)|
=\sum_{i=0}^{n}|g(i+1)-g(i)|
=\sum_{i=0}^{n}|\Delta g(i)|.
\end{equation}
This is exactly twice the right-hand side of the statement of Proposition~\ref{pro:budget}, i.e., of 
\begin{equation} \label{eq:budget-var}
C(g)
=\frac{1}{2}\sum_{i=0}^{n}|\Delta g(i)|.
\end{equation}

Comparing with existing literature, this is not new as we explain in the following.

In \cite[p. 69]{Har10} Harman noted that a function $f\colon [0,1)\to\IR$ with at most a finite number of discontinuities and finite total variation $V(f)$ satisfies
\begin{equation} \label{eq:harman}
V^*(f)
=\frac{1}{2}V(f).
\end{equation}
Here, the quantity $V^*(f)$ is a one-dimensional instance of another concept of variation which Harman introduced already in his book \cite[p. 162]{Har98}. It is defined as the Riemann integral (whenever it exists)
\begin{equation} \label{eq:harman-def}
V^*(f)
=\int_{\inf f}^{\sup f}K(f,\alpha)\dd \alpha,
\end{equation}
where $K(f,\alpha)$ is the smallest number of intervals $I_1,\dots,I_r$ such that the superlevel set at level $\alpha\in\IR$ can be expressed as
\[
\{x\in [0,1)\colon f(x)\ge \alpha\}
=\Big(\bigcup_{j=1}^{m-1} I_j \Big)\setminus \Big(\bigcup_{j=m}^r I_j\Big).
\]

An inspection of the proof of Proposition~\ref{pro:budget-general} shows that the string of equalities in \eqref{eq:discrete-harman} amounts to a discrete version of \eqref{eq:harman}.  More precisely, the numbers
\[
M_{\ell}
=\#\{i\in\ZZ\colon f(i-1)<\ell\le f(i)\},\quad \ell\in\IN,
\]
are discrete counterparts of $K(f^{\rm ext},\alpha), \alpha\in\IR,$ and at least for nonnegative $f$ we have
\[
V^*(f^{\rm ext})
=\sum_{\ell=1}^{\sup f} M_{\ell}
\]
which is a discrete version of the integral in \eqref{eq:harman-def}.  A straightforward modification is necessary to accommodate also non-nonnegative $f$. 

Harman \cite{Har10} provides no formal proof of the equality \eqref{eq:harman} and states that it follows via a limit argument from the Banach indicatrix theorem which appeared in \cite{Ban25}. The latter states that equation \eqref{eq:harman} holds for continuous $f$ when $K(f,\alpha)$ is replaced by $\frac{1}{2}\#\{x\in [0,1)\colon f(x)=\alpha\}$, which is known as the Banach indicatrix at $\alpha\in\IR$.

Let us describe another result in this direction which is due to Pausinger and Svane \cite{PS15}. In Theorem 5.8 in \cite{PS15} the Banach indicatrix is used to prove that for $f\colon [0,1]\to\IR$ with $f(0)=f(1)=0$ then
\begin{equation} \label{eq:K-var-bound}
V_{\mathcal{K}}(f)
\le \frac{1}{2}V(f),
\end{equation}
where $V_{\mathcal{K}}(f)$ is the $\mathcal{K}$-variation of $f$. To define this concept, let $\mathcal{K}$ be the collection of convex subsets of $[0,1]$, that is, all the subintervals. Then for a simple function
\begin{equation} \label{eq:dvar-rep}
f=c_0\bsone_{[0,1]}+\sum_{j=1}^{m}c_j\bsone_{K_j},\quad \text{where}\quad c_j\in\IR\text{ and }K_j\in \mathcal{K}\setminus\big\{[0,1]\big\}\text{ for }j\in [m],
\end{equation}
the $\mathcal{K}$-variation is
\[
V_{\mathcal{K}}(f)
=\inf\Big\{\sum_{j=1}^{m}|c_j|\colon f=c_0\bsone_{[0,1]}+\sum_{j=1}^{m}c_j\bsone_{K_j}\text{ with }c_j,K_j \text{ as in \eqref{eq:dvar-rep}}\Big\}.
\]
This concept of variation closely resembles $C(f)$ as in \eqref{eq:budget} and in this way \eqref{eq:K-var-bound} corresponds to the upper bound in Proposition~\ref{pro:budget}. Note however that constants can be added to a function without changing its $\mathcal{K}$-variation. More precisely, the authors of \cite{PS15} define $\mathcal{D}$-variation for functions uniformly approximable by simple functions and for a general system $\mathcal{D}$ instead of $\mathcal{K}$. For example, one can take the system of axis-parallel boxes in the $d$-dimensional unit cube $[0,1]^d$. The related $\mathcal{G}$-variation was introduced in \cite{Kur97} to study neural networks and similar concepts appear in nonlinear approximation such as in \cite[Chapter 8]{Dev98}.

In fact, the Harman variation $V^*(f)$ from \cite{Har10} and the $\mathcal{K}$-variation $V_{\mathcal{K}}(f)$ from \cite{PS15} were introduced with the purpose of defining a suitable concept of variation encapsulating the complexity of a function with respect to numerical integration. In the remaining paragraphs, we want to give more details on this motivation.

In order to approximate the integral $\int_{[0,1]^d}f(x)\dd x$ of a function $f\colon [0,1]^d\to\IR$, it is convenient to use a quasi-Monte Carlo rule 
\[
\frac{1}{n}\sum_{j=1}^{n}f(x_j),\quad \text{where}\quad n\in\IN \text{ and }x_1,\dots,x_n\in [0,1]^d.
\]
The error of such a rule can be estimated via the Koksma-Hlawka inequality (see, e.g., Drmota and Tichy \cite[Theorem 1.14 and Remark 1]{DT97} and Kuipers and Niederreiter \cite[Chapter 2, Theorem 5.5]{KN74}) which states that 
\begin{equation} \label{eq:koksma-hlawka}
\Big|\frac{1}{n}\sum_{j=1}^{n}f(x_j)-\int_{[0,1]^d}f(x)\dd x\Big|
\le D(x_1,\dots,x_n)V_{\rm HK}(f),
\end{equation}
where the first factor 
\begin{equation} \label{eq:star-disc}
D(x_1,\dots,x_n)
=\sup_{R\in \mathcal{R}^*} \Big|\frac{1}{n}\sum_{j=1}^{n}\bsone_{R}(x_j)-\int_{[0,1]^d}\bsone_R(x)\dd x\Big|
\end{equation}
is the star-discrepancy of $\{x_1,\dots,x_n\}$. Here, $\mathcal{R}^*$ is the collection of all axis-parallel boxes in $[0,1]^d$ with one endpoint in $0$ and the star-discrepancy returns the maximal deviation over all $R\in\mathcal{R}^*$ between the fraction of points $x_1,\dots,x_n$ in $R$ from the volume of $R$. 

The second factor $V_{\rm HK}(f)$ is the Hardy-Krause variation of $f$, which coincides with the total variation on $[0,1]$ if $d=1$. If $d\ge 2$, it combines its variations in a complicated way along faces of cubes. Applied to a function $f$ the Koksma-Hlawka inequality \eqref{eq:koksma-hlawka} gives a finite bound if and only if $V_{\rm HK}(f)$ is bounded. For $d\ge 2$ this need not be the case even if $f$ is smooth, for example when its support is a ball. This defect motivated the definition of the Harman variation and the subsequent $\mathcal{D}$-variation as multivariate concepts of variation. 

\begin{remark}
The concepts of multi-dimensional variations introduced above motivate a multivariate extension of our original problem. Namely, one could consider a combination lock with ``dials'' placed on an $n$ by $n$ grid where dials can be ``rotated'' simultaneously if they form a rectangle or a convex set. Also, it is possible to place the dials on a circle and obtain a periodic version of our original problem which is then related to a periodic notion of variation. 
\end{remark}

On a related note, Aistleitner et al. in \cite{APS+17} have shown that $V_{\rm HK}(f)$ coincides with $V_{\mathcal{R}^*}(f)$, which can be seen as a generalization of the equivalence between $V(f)$ and $V_{\mathcal{K}}(f)$ in dimension one. Further, Harman \cite{Har10} gave an extension of the Koksma-Hlawka inequality \eqref{eq:koksma-hlawka} by replacing $V_{\rm HK}(f)$ with $V^*(f)$ and the star-discrepancy by the so called isotropic discrepancy which is just \eqref{eq:star-disc} with $\mathcal{R}^*$ replaced by the system $\mathcal{K}$ of convex sets contained in $[0,1]^d$, see also Zaremba \cite{Zar70}. Interested readers are invited to enter the domain of discrepancy theory via the illustrative book \cite{Mat99} by Matou\v sek.

\subsection*{Acknowledgment.}
	The author wishes to acknowledge the support of the Austrian Science Fund (FWF) through Project F5513-N26, which was a part of the Special Research Program \textit{Quasi-Monte Carlo Methods: Theory and Applications}, and Project P32405 \textit{Asymptotic geometric analysis and applications}.

\end{document}